\documentclass[a4paper, 11 pt, twoside]{amsart}
\usepackage{srollens-en}[2011/03/03]
\usepackage[left = 3.5cm, right = 3.5cm, headsep = 6mm,
footskip = 10mm, top = 35mm, bottom = 35mm, footnotesep=5mm, headheight =
2cm]{geometry}

\usepackage[unicode,bookmarks, pdftex]{hyperref}
\hypersetup{colorlinks=true,citecolor=NavyBlue,linkcolor=NavyBlue,urlcolor=Orange, pdfpagemode=UseNone, breaklinks=true}

{\theoremstyle{Lehn-up}

}

\newcommand{\Diff}{\mathrm{Diff}}

\newcommand{\infect}{\!\succ\!}

\title[Virus infections and Corona surfaces]{Virus infections, Corona surfaces, and extra  components in the moduli space of stable surfaces}

\author{S\"onke Rollenske}
\address{S\"onke Rollenske\\FB 12/Mathematik und Informatik\\
Philipps-Universit\"at Marburg\\
Hans-Meerwein-Str. 6\\
35032 Marburg\\
Germany}
\email{rollenske@mathematik.uni-marburg.de}

\begin{document}
\begin{abstract}
We construct examples of non-smoothable stable surfaces, that we call Corona surfaces, with invariants in a wide range comprising all possible invariants of smooth minimal surfaces of general type but non-standard second plurigenus. 

We deduce that the moduli space of stable surfaces $\overline \gothM_{k,l}$   has  least $k-l+2$ connected components distinguished by the second plurigenus and, in particular, always has more irreducible components than the Gieseker moduli space with the same invariants.

Disclaimer: While the language and the pictures in this note are inspired by current events, the actual content is a result of pure mathematics not related to medical research.
\end{abstract}

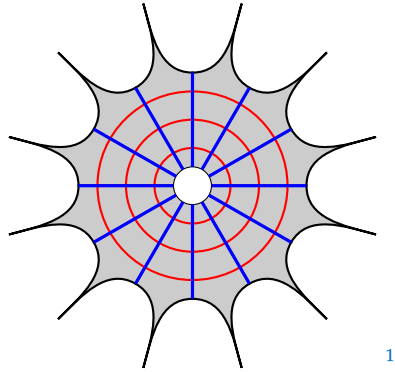
\maketitle
\setcounter{tocdepth}{1}

\begin{center}
 \begin{tikzpicture}[scale = .25,
exceptional/.style = {red, thick}, 
 nonnormal/.style ={very thick, blue}]
\def\order{12}
 \foreach \n in {1,...,\order}
 {
\draw[thick, rotate = \n*360/\order, black, fill = black!20!white] 
(0:1cm) 
to (0:6cm) 
  to[out = 90, in =180/\order+180] 
(180/\order:10cm +12/\order cm)
 to[out = 180/\order+180, in = 360/\order-90]
 (360/\order:6cm)
 to (360/\order:1cm)
;
}
\draw[fill = white]  (0,0) circle[radius = 1] ;
\draw[exceptional]  (0,0) circle[radius = 2] ;
\draw[exceptional]  (0,0) circle[radius = 3.5] ;
\draw[exceptional]  (0,0) circle[radius = 5] ;
\foreach \n in {1,...,\order}
 {
 \draw[rotate = \n*360/\order, nonnormal] (0:1cm) to (0:6cm);
 }
\end{tikzpicture}
\footnote{Corona surface $X$ with 12 irreducible virus components ($K_X^2 = 4$ and $\chi(X) = 3$).}

\end{center}

\tableofcontents

\section{Introduction}
The moduli space of stable surfaces $\overline{\gothM}_{k,l}$ is a natural compactification  of Gieseker's moduli space of canonical models of surfaces of general type with $k=K_X^2$ and $l= \chi(\ko_X)$, see \cite{Kollar13, KollarModuli} and references therein.

While it is still true that the Gieseker moduli space ${\gothM}_{k,l}$ is an open subset of $\overline{\gothM}_{k,l}$, the complement is no longer a divisor as in the moduli space of stable curves, and in many cases extra irreducible or even connected components have been found,e.g., \cite{fprr20} or \cite[5.3.3]{LR16}.

The purpose of this note is to show that this is not an accident. Recall that the invariants $K^2$ and $\chi$ of canonical models of surfaces of general type are restricted as depicted in Figure \ref{fig: geography} (see \cite[Ch.~VII]{BHPV}).

 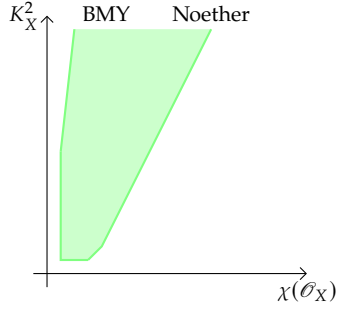
\begin{figure}[ht]
\scriptsize
\begin{tikzpicture}
[scale=.18, 
axes/.style={},
classical/.style={thick, green!50!white},
stable/.style={thick},
pt/.style={circle,draw, fill=black, size=1mm}
]

\begin{scope}[stable]
\begin{scope}[classical]
\fill[green!20!white] (12, 18) -- (4, 2) -- (3,1) -- (1,1) -- (1,9) -- (2, 18) -- cycle;
\draw (3, 1) -- (4,2) -- (12, 18) ;
\draw (1, 9) -- (2,18);

 \draw (3, 1)--(1,1) -- (1,9); 
\end{scope}
 \draw (12, 18)  node[above] {Noether};
\draw (2,18) node[above right] {BMY};
\end{scope}
\begin{scope}[axes]
\draw[->]  (-1,0)--(19,0) node[below] {$\chi(\ko_X)$};
\draw[->]  (0,-1)--(0,19) node[left]  {$K_X^2$};
\end{scope}
\end{tikzpicture}
 \caption{The geography of minimal surfaces of general type}\label{fig: geography}
\end{figure}
We show
\begin{custom}[Theorem A]
Let $(k,l)\in \IN\times \IN$ be a pair of positive integers allowed by the Noether-inequality $k\geq 2l-6$. Then there exists a stable surface $Z_{k,l}$ with invariants $K_{Z_{k,l}}^2 = k$ and $\chi(Z_{k,l}) = l$ which does not admit a $\IQ$-Gorenstein smoothing in the large. 
In particular, the moduli space of stable surfaces $\overline{\gothM}_{k,l}$ has more irreducible components than the Gieseker moduli space ${\gothM}_{k,l}$ and extra connected components.
\end{custom}

More precisely, our examples show, that for $l\geq 0$ the moduli space of stable surfaces  $\overline{\gothM}_{k,l}$ has at least $k-l+2$ connected components distinguished by the second plurigenus. 

We phrase our construction in terms of a spiky virus surface infecting honest, locally smoothable Gorenstein surfaces constructed in \cite{LR16}. For certain configurations this gives suggestive Figures as above, which is why we call these surfaces Corona surfaces.

\subsection*{Acknowledgements}
I would like to thank Diana Torres for stimulating discussions about quotient singularities and volumes and my collaborators Wenfei Liu, Marco Franciosi, Rita Pardini, and Julie Rana for many past, present and hopefully future discussions on stable surfaces.

I am grateful to Stephen Coughlan for a discussion leading to Remark \ref{rem: coughlan} and from there to a strengtheing of the main result.

\section{Set-up and Koll\'ar's glueing}
We work with algebraic varieties over the complex numbers and write  $\chi(X)$ for $\chi(\ko_X)$. A canonical divisor will be denoted by $K_X$, the reflexive powers of the canonical sheaf by $\omega_X^{[m]}$.

We recall some facts about stable surfaces, the authorative reference is  \cite[Sect.~5.1--5.3]{KollarSMMP}.
 
Let $X$ be a non-normal stable surface and $\pi\colon \bar X\to X$ its normalisation. Recall that the non-normal locus $D\subset X$ and its preimage $\bar D\subset \bar X$ are pure of codimension $1$, i.\,e., curves. Since $X$ has ordinary double points at the generic points of $D$ the map on normalisations $\bar D^\nu\to D^\nu$ is the quotient by an involution $\tau$. 
Koll\'ar's glueing principle says that $X$ can be uniquely reconstructed from $(\bar X, \bar D, \tau\colon \bar D^\nu\to \bar D^\nu)$ via the following two push-out squares:
\begin{equation}\label{diagr: pushout}
\begin{tikzcd}
    \bar X \dar{\pi}\rar[hookleftarrow]{\bar\iota} & \bar D\dar{\pi} & \bar D^\nu \lar[swap]{\bar\nu}\dar{/\tau}
    \\
X\rar[hookleftarrow]{\iota} &D &D^\nu\lar[swap]{\nu}
    \end{tikzcd}
\end{equation}

We will not use the full power of the gluing result, so let us set up a special case: 
 Let $(\bar X, \bar D)$ be a log-canonical pair of dimension $2$, possibly with many components, hence $\bar D$ is a nodal curve.  Let $\bar\nu \colon \bar D^\nu \to \bar D$ be the normalisation and denote by $\bar S\subset \bar D^\nu$ the set of preimages of nodes of $D$; in other words, $\bar S$ is the conductor of $\bar \nu$ respectively the support of the different $\Diff_{\bar D^\nu}(0)$. 
 \begin{prop}\label{prop: easy glueing}
 Assume that $\bar X$ is smooth near $\bar D$ and let $\tau$ be an involution on the normalisation $\bar D^\nu$ which preserves $\bar S$.  Denote by $\rho$ the number of fixed points of $\tau$.  
 Let  $\bar \mu = \frac 12 |S|$ be the number of nodes of $\bar D$ and let   $\delta$ 
  be the number of equivalence classes of the equivalence relation  on $\bar S$ generated by $x\sim y$ if $y = \tau x$ or $\bar\nu(x) = \bar \nu(y)$, which do not contain a fixed point of $\tau$. 
  
  Then there exists a stable surface $X$ with normalisation $\bar X$ as in diagram \eqref{diagr: pushout} with invariants $K_X^2 = (K_{\bar X}+\bar D)^2$ and
  \[ \chi(X) = \chi(\bar X) -\frac 12 \chi(\bar D) -\frac12 \bar \mu + \delta+\frac\rho4.\]
  Near the conductor $D$, the surface $X$ is $2$-Gorenstein  and Gorenstein if  $\tau|_S$ is fixed point-free. It has has $\delta$ degenerate cusps. 
 \end{prop}
 \begin{proof}
  Locally near $\bar D$ we have a nodal curve on a smooth surface. Thus the condition on the involution of preserving the preimages of nodes is exactly the 
  condition of preserving the different and Koll\'ar's glueing result \cite[Thm.~5.13]{KollarSMMP} applies. Counting the number of degenerate cusps and determining the Cartier-index near $D$ follows from the classification of slc singularities via their semi-resolutions as in \cite[Addendum, p.1535]{FPR15a}.
  
  The formulas for the invariants are contained in \cite[Prop.~3.3, Cor.~3.5] {FPR15a}.
 \end{proof}

\begin{rem}
 By our assumptions, each of the degnerate cusps is, locally analytically, of a particularly easy form, depending only on the number of nodes mapping to it: it is either the cone over a plane nodal cubic, a hypersurface with equation $z^2-x^2y^2=0$ or a cone over a cycle $n$ independent lines in $\IP^{n-1}$ (see
 \cite{LR16}).
\end{rem}

 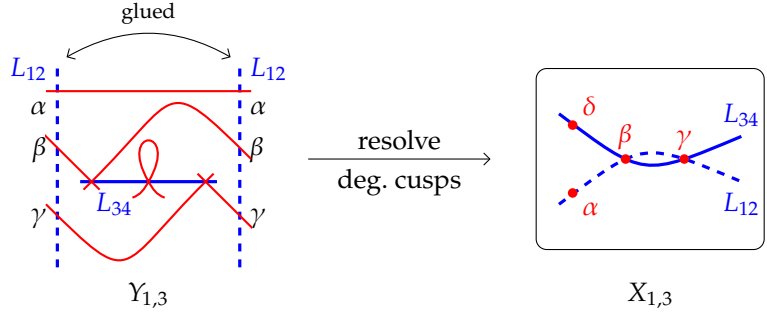
\begin{figure}[h!]\caption{The surfaces $Y_{1,3}$ and $X_{1,3}$, which has $4$ degenerate cusps.}\label{fig: X_13}
\small
 \begin{tikzpicture}
[curves/.style = {thick},
exceptional/.style = {red, thick}, 
nonnormal/.style ={very thick, blue},
boundary/.style={nonnormal, dashed},
scale = .6]

\begin{scope}[curves]
\draw[boundary] (-2,.5) node[left] {$L_{12}$}  -- (-2, -4);
\draw[boundary] (2,.5) node[right] {$L_{12}$} -- (2, -4);
\draw[nonnormal] (-1.5, -2)-- (-.75, -2) node[below]{$L_{34}$} -- ( 1.5, -2);

\draw[<->, thin] (-1.8, .75) to[bend left] node[above] {\scriptsize glued} ++(3.6,0);

\node  at (0, -4.5) {$Y_{1,3}$};

\node[below left] at (-2,0) {$\alpha$};
\node[left] at (-2,-1.3) {$\beta$};
\node[left] at (-2,-2.8) {$\gamma$};
\node[below right] at (2,0) {$\alpha$};
\node[right] at (2,-1.3) {$\beta$};
\node[right] at (2,-2.8) {$\gamma$};

\begin{scope}[exceptional]
\clip (-2.25, 1) rectangle (2.25, -5);
\draw  (-2.25, 0) -- (2.25, 0);
\draw[ every loop/.style={looseness=40, min distance=40}]
 (0,-2) ++ (-45: .5cm) to[out=160, in=-60] (0, -2) to[out=120, in =60,loop] () to[out=-120, in=20] ++(225:0.5cm);

\draw ( -1.25, -2)++(-45: .25cm) --   ++(135:2cm);
\draw ( -1.25, -2)++(-135: .25cm) -- ( -1.25, -2) to[out = 45, in = 135, looseness = 2]  (2.25, -1.5);

\draw ( 1.25, -2)++(45: .25cm) -- ( 1.25, -2) to[out = -135, in = -45, looseness = 2]  (-2.25, -2.5);
\draw ( 1.25, -2)++(135: .25cm) --   ++(-45:3cm);
\end{scope}
\end{scope}

\draw[->] (3.5, -1.5) to node[above] {resolve} node [below] {deg.\ cusps} ++(4,0);

\begin{scope}[xshift = 9cm, yshift =-1.5cm,  looseness=1.5]
\draw[rounded corners] (-.5, 2) rectangle (4.5, -2);
\node at (2, -3) {$X_{1,3}$};

 \draw[nonnormal, name path = L34] (0,1) to[bend right] (4,.5) node[above] {$L_{34}$};
\draw[boundary, name path = L12] (0,-1) to[bend left] (4,-.5)node[below] {$L_{12}$};
\fill [red, name intersections={of=L12 and L34}]
(intersection-1) circle (3pt) node [above] {$\beta$}
(intersection-2) circle (3pt) node [above]{$\gamma$};
\fill[red] (.3, .75) circle (3pt) node [above right] {$\delta$}
(.3,- .75) circle (3pt) node [below right] {$\alpha$};
\end{scope}
\end{tikzpicture}
\end{figure}

 \begin{rem}\label{rem: visualisation}
 We see that the involution $\tau$ encodes, via the generated equivalence relation, the crucial information that gives the number of degenerate cusps on $X$ and thus $\chi(X)$. In order to visualise this information better, we will usually pass to a partial resolution $\sigma\colon Y \to X$, where we blow up each of the degenerate cusps, which on the normalisation $\bar X$ corresponds to taking a log-resolution of $\bar D$. We illustrate this in Figure \ref{fig: X_13}. Thus, instead of counting degenerate cusps in $X$ or corresponding equivalence classes in $S$, we count cycles of $\sigma$-exceptional rational curves in $Y$.
\end{rem}

\subsection{Plurigenera}
The constancy of the geometric genus in  a stable family is \cite[Cor. 25]{Kollar13} and the constancy of higher plurigenera is implicit in the construction of the moduli space. For lack of explicit reference we include a short proof here.
\begin{prop}\label{prop: plurigenus}
For $m\geq2$ the plurigenera $P_m(X) = h^0(X, \omega_X^{[m]})$ are constant on connected components of the moduli space of stable surface.
\end{prop}
\begin{proof}
 Let $\pi\colon \kx \to B$ be a stable family over a smooth curve. By definition \cite{Kollar13} the sheaf $\omega_{\kx/B}^{[m]}$ is flat over $B$ and commutes with base change. Thus $h^0(\kx_b, \omega_{\kx_b}^{[m]})=\chi( \omega_{\kx_b}^{[m]})$ is constant because all higher cohomology groups vanish by the generalised Kodaira vanishing (see \cite[Cor.~19]{LR14} or \cite{Fujino14}).
\end{proof}
\begin{cor}\label{cor: non-smoothable}
Let $X$ be a stable surface. 
\begin{enumerate}
 \item If $X$ is $2$-Gorenstein, then $P_2(X) = \chi(X) + K_X^2$. 
 \item If $P_2(X)\neq \chi(X) + K_X^2$ then the connected component of the moduli space of stable surfaces does not contain a surface with canonical singularities. We say, $X$ does not admit a $\IQ$-Gorenstein smoothing in the large.
\end{enumerate}
\end{cor}
\begin{proof}
 The first item follows from the Riemann-Roch formula for Cartier divisors on stable surfaces \cite[Thm.~3.1]{LR16} applied to $2K_X$ while is a direct consequence of Proposition \ref{prop: easy glueing}. 
\end{proof}

\section{Construction of the examples}\label{sect: Xkl}
Our examples build upon the examples considered in \cite{LR16}, which get infected by spiky non-smoothable and non-Gorenstein components. As explained in Remark \ref{rem: visualisation} we use the partial resolutions for visualisation to keep track of the number of degenerate cusps.
\subsection{The surfaces $X_{k,l}$}
For the convenience of the reader we recall the surfaces $X_{k,l}$ and $Y_{k,l}$ constructed in \cite{LR16} and construct some similar ones.
\subsubsection{The elementary tiles} 
For this we start with the pair $(\IP^2, L_1+L_2+L_3+L_4)$ where the $L_i$ are general lines. We want to glue  $L_3$ to $L_4$ preserving the marked points, which we do on the log-resolution for visualisation purposes.

Consider  $\tilde \IP^2$, the plane blown up in the intersection points of the lines. There are six different ways to  glue $L_3$ to $L_4$ while preserving the intersections with the exceptional divisor, which up to isomorphism (renaming $L_1$ and $L_2$) reduce to four essentially different possibilities.  These are given in Figure~\ref{fig: elementary tiles}. The surfaces have normal crossing singularities along $L_{34}$, the image of $L_3$ and $L_4$.
Note that for esthetic reasons we sometimes use a partly mirrored version of Type A in later figures. 

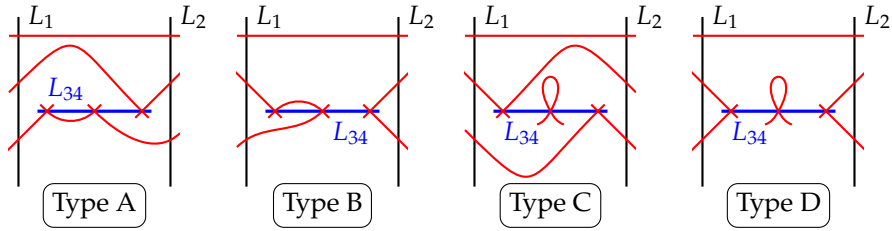
\begin{figure}
\caption{The four different possibilities to glue $L_3$ to $L_4$ up to isomorphism.}
\label{fig: elementary tiles}
\small
\begin{tikzpicture}
[curves/.style = {thick},
exceptional/.style = {red, thick}, 
nonnormal/.style ={very thick, blue},
scale=.5
]

\begin{scope}[curves, xshift = 3cm]
\draw (-2,.5) node[right] {$L_1$}  -- (-2, -4);
\draw (2,.5) node[right] {$L_2$} -- (2, -4);
\draw[nonnormal] (-1.5, -2)-- (-.75, -2) node[below]{$L_{34}$} -- ( 1.5, -2);

\node [draw, rounded corners, thin] at (0, -4.5) {Type D};

\begin{scope}[exceptional]
\clip (-2.25, 1) rectangle (2.25, -5);
\draw  (-2.25, 0) -- (2.25, 0);
\draw[ every loop/.style={looseness=40, min distance=40}]
 (0,-2) ++ (-45: .5cm) to[out=160, in=-60] (0, -2) to[out=120, in =60,loop] () to[out=-120, in=20] ++(225:0.5cm);

\draw ( -1.25, -2)++(-45: .25cm) --   ++(135:2cm);
\draw ( -1.25, -2)++(45: .25cm) --   ++(-135:3cm);

\draw ( 1.25, -2)++(-135: .25cm) --   ++(45:2cm);
\draw ( 1.25, -2)++(135: .25cm) --   ++(-45:3cm);
\end{scope}
\end{scope}

\begin{scope}[curves, xshift=-3cm]
\draw (-2,.5) node[right] {$L_1$}  -- (-2, -4);
\draw (2,.5) node[right] {$L_2$} -- (2, -4);
\draw[nonnormal] (-1.5, -2)-- (-.75, -2) node[below]{$L_{34}$} -- ( 1.5, -2);

\node [draw, rounded corners, thin] at (0, -4.5) {Type C};

\begin{scope}[exceptional]
\clip (-2.25, 1) rectangle (2.25, -5);
\draw  (-2.25, 0) -- (2.25, 0);
\draw[ every loop/.style={looseness=40, min distance=40}]
 (0,-2) ++ (-45: .5cm) to[out=160, in=-60] (0, -2) to[out=120, in =60,loop] () to[out=-120, in=20] ++(225:0.5cm);

\draw ( -1.25, -2)++(-45: .25cm) --   ++(135:2cm);
\draw ( -1.25, -2)++(-135: .25cm) -- ( -1.25, -2) to[out = 45, in = 135, looseness = 2]  (2.25, -1.5);

\draw ( 1.25, -2)++(45: .25cm) -- ( 1.25, -2) to[out = -135, in = -45, looseness = 2]  (-2.25, -2.5);
\draw ( 1.25, -2)++(135: .25cm) --   ++(-45:3cm);
\end{scope}
\end{scope}

\begin{scope}[curves, xshift = -15cm, yshift = 0cm]
\draw (-2,.5) node[right] {$L_1$}  -- (-2, -4);
\draw (2,.5) node[right] {$L_2$} -- (2, -4);
\draw[nonnormal] (-1.5, -2)-- (-.75, -2) node[above]{$L_{34}$} -- ( 1.5, -2);

\node [draw, rounded corners, thin] at (0, -4.5) {Type A};

\draw[exceptional]  (-2.25, 0) -- (2.25, 0);
\begin{scope}[exceptional, cm ={1,0,0,-1,(0,-4)}]
\clip (-2.25, 1) rectangle (2.25, -5);

\draw ( -1.25, -2)++(-135: .25cm) -- ( -1.25, -2) to[out = 45, in = 135] 
 (0, -2)-- ++(-45: .25cm);
\draw ( 0, -2)++(-135: .25cm) -- (0, -2) to[out = 45, in =135] 
 (2.25, -1.42);

\draw ( -1.25, -2)++(-45: .25cm) --   ++(135:2cm);

\draw ( 1.25, -2)++(135: .25cm) --   ++(-45:3cm);
\draw ( 1.25, -2)++(45: .25cm) -- ( 1.25, -2) to[out = -135, in = -45, looseness = 2]  (-2.25, -2.5);
\end{scope}
\end{scope}

\begin{scope}[curves, xshift = -9cm, yshift = 0cm]
\draw (-2,.5) node[right] {$L_1$}  -- (-2, -4);
\draw (2,.5) node[right] {$L_2$} -- (2, -4);
\draw[nonnormal] (-1.5, -2)-- (.75, -2) node[below]{$L_{34}$} -- ( 1.5, -2);

\node [draw, rounded corners, thin] at (0, -4.5) {Type B};

\begin{scope}[exceptional]
\clip (-2.25, 1) rectangle (2.25, -5);
\draw  (-2.25, 0) -- (2.25, 0);

\draw ( -1.25, -2)++(-135: .25cm) -- ( -1.25, -2) to[out = 45, in = 135] 
 (0, -2)-- ++(-45: .25cm);
\draw ( 0, -2)++(45: .25cm) -- (0, -2) to[out = -135, in = 45] 
 (-2.25, -2.95);

\draw ( -1.25, -2)++(-45: .25cm) --   ++(135:2cm);

\draw ( 1.25, -2)++(-135: .25cm) --   ++(45:2cm);
\draw ( 1.25, -2)++(135: .25cm) --   ++(-45:3cm);

\end{scope}
\end{scope}

\end{tikzpicture}
\end{figure}

\subsubsection{Glueing together in a circle}
We now glue the elementary tiles in a circle such that the number of degenerate cusps varies in a big enough range to obtain the stable surfaces $X_{k, l}$. For better visualisation we work with the partial resolution, so first have a look at Figure \ref{fig: Y32}, where the partial resolution, in this case actually the semi-resolution, of $X_{3,2}$ is depicted; it is glued from two tiles of type $B$ and one tile of type $D$. One can directly see the five cycles of exceptional curves on $Y_{3,2}$, leading to five degenerate cusps in $X_{3,2}$. 

 \begin{figure}[hb!]
  \caption{The semi-resolution $Y_{3,2}$ of $X_{3,2}$}\label{fig: Y32}
\begin{tikzpicture}[curves/.style = {thick},
exceptional/.style = {red, thick}, 
nonnormal/.style ={very thick, blue},
boundary/.style={nonnormal, dashed},
scale = .25
]

 \begin{scope}
 \def\order{3} \node at (120/\order : 10cm +12/\order cm) {$Y_{3,2}$};
 \foreach \n in {1,2,3} {
 \begin{scope}[rotate = \n*360/\order]
\draw[thick, black] 
(0:1cm) 
to (0:6cm) 
  to
(180/\order:10cm +12/\order cm)
 to
 (360/\order:6cm)
 to (360/\order:1cm)
  arc [start angle = 360/\order, end angle = 0, radius = 1cm]
;
\foreach \r in {2cm} {
\draw[exceptional]  (0:\r) arc[radius = \r, start angle = 0 , delta angle = 360/\order] ;
}
 \end{scope}
 }
  
 \foreach \n in {1} {
 \begin{scope}[rotate = \n*360/\order]

\begin{scope}[rotate = -90+180/\order, yshift = 3.88cm] 
\draw[exceptional,  every loop/.style={looseness=80, min distance=80}]
 (0,0) ++ (-45: .5cm) to[out=160, in=-60] (0, 0) to[out=120, in =60,loop] () to[out=-120, in=20] ++(225:0.5cm);
 \end{scope}
 
 \coordinate (a) at (90/\order: 4.45); 
 \draw[exceptional] (0:3.5) to (a);
\draw[exceptional] (0:5) to (a);

 \coordinate (b) at (270/\order: 4.45); 
 \draw[exceptional] (360/\order:3.5) to (b);
\draw[exceptional] (360/\order:5) to (b);

 \draw[nonnormal] (300/\order:5cm) to (60/\order:5cm);
\draw[nonnormal] (0:1cm) to (0:6cm);
 \draw[nonnormal] (360/\order:1cm) to (360/\order:6cm);

 \node at (180/\order: 8) {$D$};
\end{scope}
 }

 \foreach \n in {2,3} {
 \begin{scope}[rotate = \n*360/\order]

 \coordinate (a) at (90/\order: 4.45); 
 \coordinate (b) at (270/\order: 4.45); 
 \coordinate (c) at (180/\order: 3.88); 
 
 \draw[exceptional] (0:3.5) to (a);
\draw[exceptional] (0:5) to (a);

\draw[exceptional] (b) to[bend left] (c);

 \draw[exceptional] (360/\order:3.5) to[bend right] (c);
\draw[exceptional] (360/\order:5) to (b);

 \draw[nonnormal] (300/\order:5cm) to (60/\order:5cm);
\draw[nonnormal] (0:1cm) to (0:6cm);
 \draw[nonnormal] (360/\order:1cm) to (360/\order:6cm);

 \node at (180/\order: 8) {$B$};
\end{scope}
} 

\end{scope}
\end{tikzpicture}
\end{figure}
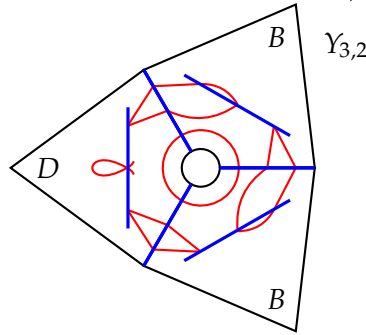

We give the other partial resolutions $Y_{k,l}$ in Figure \ref{fig: chi small} and in Figure \ref{fig: chi big}.

Recall from \cite{LR16} that $X_{k,l}$ is a locally smoothable Gorenstein stable surface with invariants
\[ K_{X_{k,l}}^2 = k \text{ and } \chi(X_{k,l}) = l.\]
The invariants can also be computed easily from Proposition \ref{prop: easy glueing}.

\begin{figure}[h!]\caption{The surface $Y_{k,l}$ for $1-k<l\leq 1 $ glued from $1-l$ tiles of type $A$ and $k+l-1$ tiles of type $B$. }\label{fig: chi small}
\scriptsize
 \begin{tikzpicture}
[curves/.style = {thick},
exceptional/.style = {red, thick}, 
nonnormal/.style ={very thick, blue},
boundary/.style={nonnormal, dashed},
scale = .5
]

\begin{scope}[curves]
\draw[boundary] (-16,.5) node[right] {$L_{12}$}  -- ++(0, -4.5);
\draw[boundary] (8,.5) node[right] {$L_{12}$} -- ++(0, -4.5);
\foreach \x in {-15.5, -9.5,-5.5,  -1.5,4.5} {\draw[nonnormal] (\x, -2) -- ++(3,0);};
\foreach \x in {-12, -10, -6, -2,2, 4} {\draw[nonnormal] ( \x, .5) -- ++(0, -4.5);};
\end{scope}

\begin{scope}[exceptional, thin, dotted]
\foreach \y in {0, -1.3,-2.8} {\draw( 1.75, \y) -- ++(2.5, 0);};
\draw( -12, 0) -- ++(4, 0);
\draw ( -12, -1.3)++(135: .25cm) -- ( -12, -1.3) to[out = -45, in = 45] 
 (-12, -2.8)-- ++(-135: .25cm);
\draw ( -10, -1.3)++(45: .25cm) -- ( -10, -1.3) to[out = -135, in = 135] 
 (-10, -2.8)-- ++(-45: .25cm);
\node[black] at (-11, -2) {\dots}; 
\end{scope}

\begin{scope}[xshift=-14cm]
\node[below left] at (-2,0) {$\alpha$};
\node[left] at (-2,-1.3) {$\beta$};
\node[left] at (-2,-2.8) {$\gamma$};
\end{scope}
\begin{scope}[xshift = 6cm]
\node[below right] at (2,0) {$\alpha$};
\node[right] at (2,-1.3) {$\beta$};
\node[right] at (2,-2.8) {$\gamma$};
\end{scope}

\begin{scope}[exceptional, xshift = -14cm]
\clip (-2.25, 1) rectangle (2.25, -5);
\draw  (-2.25, 0) -- (2.25, 0);

\draw ( -1.25, -2)++(-135: .25cm) -- ( -1.25, -2) to[out = 45, in = 135]  (0, -2)-- ++(-45: .25cm);
\draw ( 0, -2)++(45: .25cm) -- (0, -2) to[out = -135, in = 45]  (-2.25, -2.95);

\draw ( -1.25, -2)++(-45: .25cm) --   ++(135:2cm);

\draw ( 1.25, -2)++(-135: .25cm) --   ++(45:2cm);
\draw ( 1.25, -2)++(135: .25cm) --   ++(-45:3cm);

\end{scope}

\begin{scope}[exceptional, xshift  =-8cm]
\clip (-2.25, 1) rectangle (2.25, -5);
\draw  (-2.25, 0) -- (2.25, 0);

\draw ( -1.25, -2)++(-135: .25cm) -- ( -1.25, -2) to[out = 45, in = 135] 
 (0, -2)-- ++(-45: .25cm);
\draw ( 0, -2)++(45: .25cm) -- (0, -2) to[out = -135, in = 45] 
 (-2.25, -2.95);

\draw ( -1.25, -2)++(-45: .25cm) --   ++(135:2cm);

\draw ( 1.25, -2)++(-135: .25cm) --   ++(45:2cm);
\draw ( 1.25, -2)++(135: .25cm) --   ++(-45:3cm);

\end{scope}
\begin{scope}[exceptional, xshift =-4cm]
\clip (-2.25, 1) rectangle (2.25, -5);
\draw  (-2.25, 0) -- (2.25, 0);

\draw ( -1.25, -2)++(-135: .25cm) -- ( -1.25, -2) to[out = 45, in = 135] 
 (0, -2)-- ++(-45: .25cm);
\draw ( 0, -2)++(45: .25cm) -- (0, -2) to[out = -135, in = 45] 
 (-2.25, -2.95);

\draw ( -1.25, -2)++(-45: .25cm) --   ++(135:2cm);

\draw ( 1.25, -2)++(-135: .25cm) --   ++(45:2cm);
\draw ( 1.25, -2)++(135: .25cm) --   ++(-45:3cm);

\end{scope}

\draw[exceptional]  (-2.25, 0) -- (2.25, 0);
\begin{scope}[exceptional, cm ={1,0,0,-1,(0,-4)}]
\clip (-2.25, 1) rectangle (2.25, -5);

\draw ( -1.25, -2)++(-135: .25cm) -- ( -1.25, -2) to[out = 45, in = 135] 
 (0, -2)-- ++(-45: .25cm);
\draw ( 0, -2)++(-135: .25cm) -- (0, -2) to[out = 45, in =135] 
 (2.25, -1.42);

\draw ( -1.25, -2)++(-45: .25cm) --   ++(135:2cm);

\draw ( 1.25, -2)++(135: .25cm) --   ++(-45:3cm);
\draw ( 1.25, -2)++(45: .25cm) -- ( 1.25, -2) to[out = -135, in = -45, looseness = 2]  (-2.25, -2.5);
\end{scope}
\begin{scope}[exceptional, xshift =6cm]
\clip (-2.25, 1) rectangle (2.25, -5);
\draw  (-2.25, 0) -- (2.25, 0);

\draw ( -1.25, -2)++(-135: .25cm) -- ( -1.25, -2) to[out = 45, in = 135] 
 (0, -2)-- ++(-45: .25cm);
\draw ( 0, -2)++(-135: .25cm) -- (0, -2) to[out = 45, in =135] 
 (2.25, -1.42);

\draw ( -1.25, -2)++(-45: .25cm) --   ++(135:2cm);

\draw ( 1.25, -2)++(135: .25cm) --   ++(-45:3cm);
\draw ( 1.25, -2)++(45: .25cm) -- ( 1.25, -2) to[out = -135, in = -45, looseness = 2]  (-2.25, -2.5);
\end{scope}

\end{tikzpicture}
\end{figure}
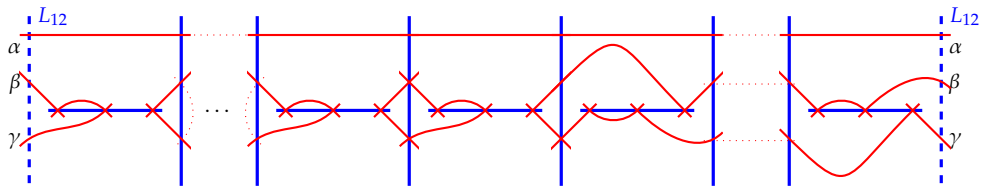

\begin{figure}[h!]\caption{The surface $Y_{k,l}$ for $2\leq l\leq k+1 $ glued from $l-1$ tiles of type $D$ and $k-l+1$ tiles of type $B$. }\scriptsize\label{fig: chi big}
 \begin{tikzpicture}
[curves/.style = {thick},
exceptional/.style = {red, thick}, 
nonnormal/.style ={very thick, blue},
boundary/.style={nonnormal, dashed},
scale = .5
]

\begin{scope}[curves]
\draw[boundary] (-16,.5) node[right] {$L_{12}$}  -- ++(0, -4.5);
\draw[boundary] (8,.5) node[right] {$L_{12}$} -- ++(0, -4.5);
\foreach \x in {-15.5, -9.5,-5.5,  -1.5,4.5} {\draw[nonnormal] (\x, -2) -- ++(3,0);};
\foreach \x in {-12, -10, -6, -2,2, 4} {\draw[nonnormal] ( \x, .5) -- ++(0, -4.5);};
\end{scope}

\begin{scope}[exceptional, thin, dotted]
\draw( -12, 0) -- ++(2, 0);
\draw ( -12, -1.3)++(135: .25cm) -- ( -12, -1.3) to[out = -45, in = 45] 
 (-12, -2.8)-- ++(-135: .25cm);
\draw ( -10, -1.3)++(45: .25cm) -- ( -10, -1.3) to[out = -135, in = 135] 
 (-10, -2.8)-- ++(-45: .25cm);
\draw( 2, 0) -- ++(2, 0);
\draw ( 2, -1.3)++(135: .25cm) -- ( 2, -1.3) to[out = -45, in = 45] 
 (2, -2.8)-- ++(-135: .25cm);
\draw ( 4, -1.3)++(45: .25cm) -- ( 4, -1.3) to[out = -135, in = 135] 
 (4, -2.8)-- ++(-45: .25cm);
\node[black] at (-11, -2) {\dots}; 
\node[black] at (3, -2) {\dots}; 

\end{scope}

\begin{scope}[xshift=-14cm]
\node[below left] at (-2,0) {$\alpha$};
\node[left] at (-2,-1.3) {$\beta$};
\node[left] at (-2,-2.8) {$\gamma$};
\end{scope}
\begin{scope}[xshift = 6cm]
\node[below right] at (2,0) {$\alpha$};
\node[right] at (2,-1.3) {$\beta$};
\node[right] at (2,-2.8) {$\gamma$};
\end{scope}

\begin{scope}[exceptional, xshift = -14cm]
\clip (-2.25, 1) rectangle (2.25, -5);
\draw  (-2.25, 0) -- (2.25, 0);

\draw ( -1.25, -2)++(-135: .25cm) -- ( -1.25, -2) to[out = 45, in = 135]  (0, -2)-- ++(-45: .25cm);
\draw ( 0, -2)++(45: .25cm) -- (0, -2) to[out = -135, in = 45]  (-2.25, -2.95);

\draw ( -1.25, -2)++(-45: .25cm) --   ++(135:2cm);

\draw ( 1.25, -2)++(-135: .25cm) --   ++(45:2cm);
\draw ( 1.25, -2)++(135: .25cm) --   ++(-45:3cm);

\end{scope}

\begin{scope}[exceptional, xshift  =-8cm]
\clip (-2.25, 1) rectangle (2.25, -5);
\draw  (-2.25, 0) -- (2.25, 0);

\draw ( -1.25, -2)++(-135: .25cm) -- ( -1.25, -2) to[out = 45, in = 135] 
 (0, -2)-- ++(-45: .25cm);
\draw ( 0, -2)++(45: .25cm) -- (0, -2) to[out = -135, in = 45] 
 (-2.25, -2.95);

\draw ( -1.25, -2)++(-45: .25cm) --   ++(135:2cm);

\draw ( 1.25, -2)++(-135: .25cm) --   ++(45:2cm);
\draw ( 1.25, -2)++(135: .25cm) --   ++(-45:3cm);

\end{scope}
\begin{scope}[exceptional, xshift =-4cm]
\clip (-2.25, 1) rectangle (2.25, -5);
\draw  (-2.25, 0) -- (2.25, 0);
\draw[ every loop/.style={looseness=40, min distance=40}]
 (0,-2) ++ (-45: .5cm) to[out=160, in=-60] (0, -2) to[out=120, in =60,loop] () to[out=-120, in=20] ++(225:0.5cm);

\draw ( -1.25, -2)++(-45: .25cm) --   ++(135:2cm);
\draw ( -1.25, -2)++(45: .25cm) --   ++(-135:3cm);

\draw ( 1.25, -2)++(-135: .25cm) --   ++(45:2cm);
\draw ( 1.25, -2)++(135: .25cm) --   ++(-45:3cm);
\end{scope}

\begin{scope}[exceptional,]
\clip (-2.25, 1) rectangle (2.25, -5);
\draw  (-2.25, 0) -- (2.25, 0);
\draw[ every loop/.style={looseness=40, min distance=40}]
 (0,-2) ++ (-45: .5cm) to[out=160, in=-60] (0, -2) to[out=120, in =60,loop] () to[out=-120, in=20] ++(225:0.5cm);

\draw ( -1.25, -2)++(-45: .25cm) --   ++(135:2cm);
\draw ( -1.25, -2)++(45: .25cm) --   ++(-135:3cm);

\draw ( 1.25, -2)++(-135: .25cm) --   ++(45:2cm);
\draw ( 1.25, -2)++(135: .25cm) --   ++(-45:3cm);
\end{scope}
\begin{scope}[exceptional, xshift =6cm]
\clip (-2.25, 1) rectangle (2.25, -5);
\draw  (-2.25, 0) -- (2.25, 0);
\draw[ every loop/.style={looseness=40, min distance=40}]
 (0,-2) ++ (-45: .5cm) to[out=160, in=-60] (0, -2) to[out=120, in =60,loop] () to[out=-120, in=20] ++(225:0.5cm);

\draw ( -1.25, -2)++(-45: .25cm) --   ++(135:2cm);
\draw ( -1.25, -2)++(45: .25cm) --   ++(-135:3cm);

\draw ( 1.25, -2)++(-135: .25cm) --   ++(45:2cm);
\draw ( 1.25, -2)++(135: .25cm) --   ++(-45:3cm);
\end{scope}

\end{tikzpicture}
\end{figure}
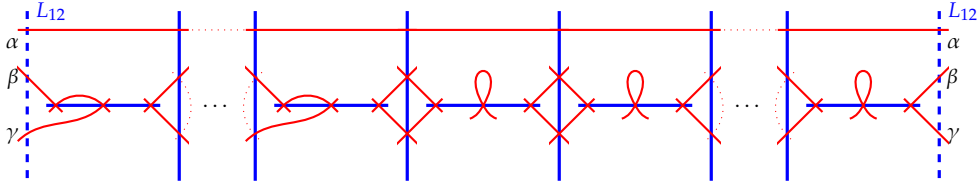

\subsubsection{Glueing together on a stick}

Let us construct a family of cousins $X_{k, k+2}$ of the surfaces $X_{k,l}$, where instead of glueing in a circle we pinch two of the lines. In other words, the involution $\tau$ will fix two of the lines $L_1$ and $L_2$.
Then $\tau|_{L_i}$ has two fixed points and, since $|S\cap L_i|=3$ and $\tau$ has to preserve this intersection, one of these fixed points is the preimage of a node.
In our picture of the partial resolution $Y_{k, k+2}$ in Figure \ref{fig: on a stick} we will thus see one string of rational curves ending in  pinch points, which then contracts to a $\IZ/2$-quotient of a degenerate cusp by the classification of slc singularities (compare \cite{KollarSMMP}).
Counting the remaining cycles of rational curves we see that $X_{k, k+2}$
has $2k+1$ degenerate cusps and from Proposition \ref{prop: easy glueing} we compute
 $K_{X_{k, k+2}}^2 = k$ and
  \[ \chi(X_{k, k+2}) =  k\cdot 1 - \frac k2 \cdot(-2) -\frac k2 \cdot 6 +(2k+1)+\frac44=k+2,\]
as suggested by the notation.

\begin{figure}[ht!]\caption{The surface $Y_{k, k+2}$ glued from $k$ tiles of type $D$; the lines $L_1$ and $L_2$ are pinched. }
\label{fig: on a stick}
\scriptsize
 \begin{tikzpicture}
[curves/.style = {thick},
exceptional/.style = {red, thick}, 
nonnormal/.style ={very thick, blue},
boundary/.style={nonnormal, dashed},
scale = .5
]

\begin{scope}[curves]
\foreach \x in { -9.5,  -1.5} {\draw[nonnormal] (\x, -2) -- ++(3,0);};
\foreach \x in {-10,  -6, -2,2} {\draw[nonnormal] ( \x, .5) -- ++(0, -4.5);};
\draw[nonnormal, thin] (-5.5, -2) -- ++(3,0);
\end{scope}

\begin{scope}[exceptional, xshift  =-8cm]
\clip (-2.25, 1) rectangle (2.25, -5);
\draw  (-2.25, 0) -- (2.25, 0);
\draw[ every loop/.style={looseness=40, min distance=40}]
 (0,-2) ++ (-45: .5cm) to[out=160, in=-60] (0, -2) to[out=120, in =60,loop] () to[out=-120, in=20] ++(225:0.5cm);

\draw ( -2, -2)++(-135: .25cm) -- ( -2, -2) to[out = 45, in = 135]  ++(1, 0)-- ++(-45: .25cm);
\draw ( -2, -2)++(135: .25cm) -- ( -2, -2) to[out = -45, in = -135]  ++(1, 0)-- ++(45: .25cm);

\draw ( 1.25, -2)++(-135: .25cm) --   ++(45:2cm);
\draw ( 1.25, -2)++(135: .25cm) --   ++(-45:3cm);
\end{scope}
\begin{scope}[exceptional, xshift =-4cm, thin, dotted]
\clip (-2.25, 1) rectangle (2.25, -5);
\draw  (-2.25, 0) -- (2.25, 0);
\draw[ every loop/.style={looseness=40, min distance=40}]
 (0,-2) ++ (-45: .5cm) to[out=160, in=-60] (0, -2) to[out=120, in =60,loop] () to[out=-120, in=20] ++(225:0.5cm);

\draw ( -1.25, -2)++(-45: .25cm) --   ++(135:2cm);
\draw ( -1.25, -2)++(45: .25cm) --   ++(-135:3cm);

\draw ( 1.25, -2)++(-135: .25cm) --   ++(45:2cm);
\draw ( 1.25, -2)++(135: .25cm) --   ++(-45:3cm);
\end{scope}
\begin{scope}[exceptional]
\clip (-2.25, 1) rectangle (2.25, -5);
\draw  (-2.25, 0) -- (2.25, 0);
\draw[ every loop/.style={looseness=40, min distance=40}]
 (0,-2) ++ (-45: .5cm) to[out=160, in=-60] (0, -2) to[out=120, in =60,loop] () to[out=-120, in=20] ++(225:0.5cm);

\draw ( -1.25, -2)++(-45: .25cm) --   ++(135:2cm);
\draw ( -1.25, -2)++(45: .25cm) --   ++(-135:3cm);

\draw ( 2, -2)++(-45: .25cm) -- ( 2, -2) to[in = 45, out = 135]  ++(-1, 0)-- ++(-135: .25cm);
\draw ( 2, -2)++(45: .25cm) -- ( 2, -2) to[in = -45, out = -135]  ++(-1, 0)-- ++(135: .25cm);
\end{scope}

\fill[black] (-10, 0) circle (3pt) node[above left] {$L_1$};
\fill[black] (2, 0) circle (3pt) node[above right] {$L_2$};

\fill[black] (-10, -3.5) circle (3pt);
\fill[black] (2, -3.5) circle (3pt) node[right] {\scriptsize pinch point};

\end{tikzpicture}
\end{figure}
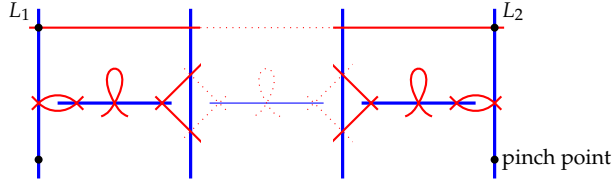

 \subsubsection{The surface $X_{2,4}$ (alternative construction)}
 We add one more surface into the mix. Take two copies of the elementary tile $(\IP^2, L_1+\dots+L_4)$ and glue them by identifying the the two configurations of lines to get a surface $X_{2,4}$. Exchanging the components defines an involution, so more concretely 
 \[ X_{2,4} = \{ z^2-(x_1x_2x_3(x_1+x_2+x_3))^2=0\}\subset \IP(1,1,1,4)\]
 is a double cover of $\IP^2$ branched over a (non-reduced) octic. It has $K_{X_{2,4}}^2 = 2$ and $\chi(X_{2,4}) = 4$. Gorenstein stable surfaces with these invariants have been studies thorougly in \cite{anthes19}.
 
\subsection{The virus surface and the infection mechanism}
We will now introduce the virus surface, which can infect the surfaces $X_{k,l}$ to make them non-$\IQ$-Gorenstein smoothable surfaces.

\subsubsection{The elementary tile $V$}
Let $V = \IP(1,1,3)$, which has a quotient singularity of type $\frac 13(1,1)$. Taking the embedding with $\ko(3)$ we realise $V\subset \IP^4$ as the cone over a rational normal curve of degree $3$. Now let $C_1, C_2$ be general hyperplane sections of $V$. Then $C_i$ is a smooth  rational curve and the two curves intersect in three points. The pair $(V, C_1+C_2)$ is a log-canonical, because the quotient singularity is log terminal and the boundary is nodal. It is $3$-Gorenstein and has invariants 
\[ (K_V+C_1+C_2)^2 = \frac 13, \quad \chi(V) = 1.\]

We depict $V$ and $\tilde V$, the blow up of $V$ at the intersection points of $C_1$ and $C_2$, in Figure \ref{fig: V and tilde V}.
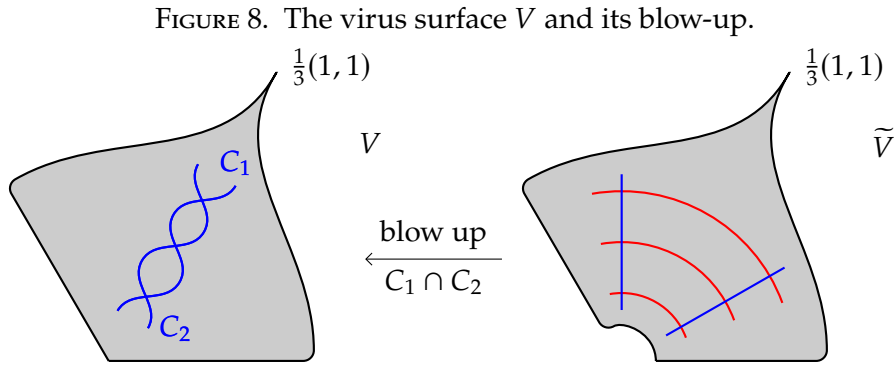
\begin{figure}[hb]
 \caption{The virus surface $V$ and its blow-up.}
 \label{fig: V and tilde V}
 \begin{tikzpicture}
[curves/.style = {thick},
exceptional/.style = {red, thick}, 
nonnormal/.style ={very thick, blue},
boundary/.style={nonnormal, thick},
scale = .45
]

\begin{scope}%
 \def\order{3}
 \node at (120/\order : 10cm +12/\order cm) {$\tilde V$};
 \begin{scope}
\draw[thick, black, fill = black!20!white, rounded corners] 
(0:1cm) 
to (0:6cm) 
  to[out = 90, in =180/\order+180] 
(180/\order:10cm +12/\order cm) node[right] {$\frac13(1,1)$}
 to[out = 180/\order+180, in = 360/\order-90]
 (360/\order:6cm)
 to (360/\order:1cm)
  arc [start angle = 360/\order, end angle = 0, radius = 1cm];
 \end{scope}

\foreach \r in {2cm, 3.5cm, 5cm} 
{
\draw[exceptional]  (20:\r) arc[radius = \r, start angle = 20 , delta angle = 80] ;
}
\foreach \phi in {30,90} 
{
\draw[boundary]  (\phi:1.5) --(\phi:5.5) ;
}
 \end{scope}
\draw[->] (-3.5, 3) to node [above] {blow up} node[below]{$C_1 \cap C_2$} ++(-4, 0); 
\begin{scope}[xshift  = -15cm]
 \def\order{3}
 \node at (120/\order : 10cm +12/\order cm) {$V$};
 \begin{scope}
\draw[thick, black, fill = black!20!white, rounded corners] 
(0,0) 
to (0:6cm) 
  to[out = 90, in =180/\order+180] 
(180/\order:10cm +12/\order cm) node[right] {$\frac13(1,1)$}
 to[out = 180/\order+180, in = 360/\order-90]
 (360/\order:6cm)
 to (0,0);
 \end{scope}
\foreach \phi in {30,90} 
{\def\step{35}
\def\a{2}
\draw[boundary, every to/.style={in = 240, out = 60}]  (80:1.5)  to ++(60-\step:\a) to ++(60+\step:\a) to ++(60-\step:\a) node[above] {$C_1$};
\draw[boundary, every to/.style={in = 240, out = 60}]  (40:1.5) node[right] {$C_2$} to  ++(60+\step:\a) to ++(60-\step:\a) to ++(60+\step:\a);
}

 \end{scope}

\end{tikzpicture}
\end{figure}

\subsubsection{The infection mechanism}
Our virus surface pair $(V, C_1+C_2)$ can infect  non-normal stable surfaces which present suitable attachment points. This is easier to visualise on the partial resolution, see Figure \ref{fig: elementary infection}.
\begin{figure}[h!] 
 \caption{Elementary infection viewed on the partial resolution}\label{fig: elementary infection}
\begin{tikzpicture}[curves/.style = {thick},
exceptional/.style = {red, thick}, 
nonnormal/.style ={very thick, blue},
boundary/.style={nonnormal, dashed},
scale = .4
]

\begin{scope}[rotate = 90]
 
\begin{scope}[curves] 
\foreach \h in {1cm, 2.5cm, 4cm} 
{
\draw[exceptional]  (0,\h) to ++(-1,0) ;
\draw[exceptional, dotted]  (-1,\h) to ++(-1,0) ;
}
\draw[dotted] (-3, -1) to (0,0);
\draw  (-2, -2/3) to (0,0);
\draw[dotted] (-3, 6) to (0, 5);
\draw(-2, 17/3) to (0,5);
\draw[nonnormal] (0,0) to (0,5);
\end{scope}

\begin{scope}[curves,  cm ={-1,0,0,1,(0,0)}]
\foreach \h in {1cm, 2.5cm, 4cm} 
{
\draw[exceptional]  (0,\h) to ++(-1,0) ;
\draw[exceptional, dotted]  (-1,\h) to ++(-1,0) ;
}
\draw[dotted] (-3, -1) to (0,0);
\draw  (-2, -2/3) to (0,0);
\draw[dotted] (-3, 6) to (0, 5);
\draw(-2, 17/3) to (0,5);
\draw[nonnormal] (0,0) to (0,5);
\end{scope}

\end{scope}

\draw[->, dashed] (2,0) to node[above] {glue in $\tilde V$} ++ (4,0);

 \begin{scope}[xshift=7cm, rotate = -90]
 \def\order{12} 

 \begin{scope}[rotate = 75]
\draw[thick, black, fill = black!20!white] 
(0:1cm) 
to (0:6cm) 
  to[out = 90, in =180/\order+180] 
(180/\order:10cm +12/\order cm)
 to[out = 180/\order+180, in = 360/\order-90]
 (360/\order:6cm)
 to (360/\order:1cm)
  arc [start angle = 360/\order, end angle = 0, radius = 1cm]
;
 \draw[nonnormal] (0:1cm) to (0:6cm);
 \draw[nonnormal] (360/\order:1cm) to (360/\order:6cm);
\foreach \r in {2cm, 3.5cm, 5cm} 
{
\draw[exceptional]  (0:\r) arc[radius = \r, start angle = 0 , delta angle = 360/\order] ;
}
 \end{scope}

\begin{scope}[curves , rotate=15,   yshift = 1cm] 
\foreach \h in {1cm, 2.5cm, 4cm} 
{
\draw[exceptional]  (0,\h) to ++(-1,0) ;
\draw[exceptional, dotted]  (-1,\h) to ++(-1,0) ;
}
\draw[dotted] (-3, -1) to (0,0);
\draw  (-2, -2/3) to (0,0);
\draw[dotted] (-3, 6) to (0, 5);
\draw(-2, 17/3) to (0,5);
\draw[nonnormal] (0,0) to (0,5);
\end{scope}

\begin{scope}[curves, rotate=-15,   cm ={-1,0,0,1,(0,1)}]
\foreach \h in {1cm, 2.5cm, 4cm} 
{
\draw[exceptional]  (0,\h) to ++(-1,0) ;
\draw[exceptional, dotted]  (-1,\h) to ++(-1,0) ;
}
\draw[dotted] (-3, -1) to (0,0);
\draw  (-2, -2/3) to (0,0);
\draw[dotted] (-3, 6) to (0, 5);
\draw(-2, 17/3) to (0,5);
\draw[nonnormal] (0,0) to (0,5);
\end{scope}
 
\end{scope}
\end{tikzpicture}
 \end{figure}
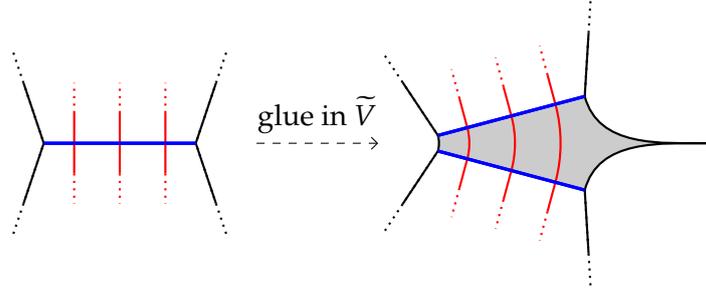

In more technical terms we can describe the elementary infection as follows:
Fix one and for all an involution $\sigma$ on the curve $C_1+C_2$ preserving the nodes. 
\begin{prop}\label{prop: infection}
 Let $X$ be a connected stable surface with the notation as in \eqref{diagr: pushout} and Proposition \ref{prop: easy glueing}.
Assume that $\bar D^\nu$ contains two components $\bar D^\nu_1\isom \IP^1\isom \bar D^\nu_2$ such that $\tau(\bar D_1) = \bar D_2$ and both contain $3$ marked points, that is, $|S\cap \bar D_i |=3$.
Now define an involution $\tilde \tau$ on $\bar D ^\nu \sqcup C_1\sqcup C_2$ by
\[ \tilde \tau|_{\bar D^\nu \setminus \bar D^\nu_1\setminus \bar D^\nu_2} = \tau, \quad \tilde \tau(\bar D^\nu_i ) = C_i\]
chosen such that $\tilde \tau$ preserves the marked points and  on $\bar D_i^\nu$ we have $\tilde\tau \circ \sigma \circ\tilde \tau = \tau$. 

Then $(\bar X \sqcup V, \bar D\sqcup C_1+C_2, \tilde \tau)$ defines a connected stable surface with the same number $\mu_1$ of degenerate cusps as $X$. We denote this surface with $V\infect X$. 
\end{prop}
\begin{proof}
 All claims are selfevident when looking at the pictures of the partial resolution, because each exceptional curve on $V$ just increases the length of one of the existing cycles or chains by $1$. Note that, by our choice of $\tilde \tau$, no new connections between cycles are created.
\end{proof}
\begin{rem}\label{rem: choices}
 If we  use elementary infection $a$ times we will denote the resulting surface with $aV\infect X$.
 In each step there are possibly many places where the infection can take place giving rise to different surface, as  illustrated by Figure \ref{fig: choices}.
 Nevertheless, for our purpose, these choices are irrelevant. \end{rem}

 \begin{figure}[h!b]
  \caption{Two different triple  infections $3V\infect X_{3,2}$ of $X_{3, 2}$ (compare Figure \ref{fig: Y32}).}
\label{fig: choices}
  \begin{tikzpicture}[curves/.style = {thick},
exceptional/.style = {red, thick}, 
nonnormal/.style ={very thick, blue},
boundary/.style={nonnormal, dashed},
scale = .3
]

 \begin{scope}[yshift = -15cm, xshift = -12 cm]
\def\aa{12}
 \def\order{4}
 \foreach \n in {1,2,3} {
 \begin{scope}[rotate = \n*480/\order]
\draw[thick, black] 
(0:1cm) 
to (0:6cm) 
  to
(180/\order:10cm)
 to
 (360/\order:6cm)
 to (360/\order:1cm)
  arc [start angle = 360/\order, end angle = 0, radius = 1cm]
;
\foreach \r in {2cm} {
\draw[exceptional]  (0:\r) arc[radius = \r, start angle = 0 , delta angle = 360/\order] ;
}
 \end{scope}
 }
  
 \foreach \n in {1} {
 \begin{scope}[rotate = \n*480/\order]

\begin{scope}[rotate = -90+180/\order, yshift = 4.35cm] 
\draw[exceptional,  every loop/.style={looseness=80, min distance=80}]
 (0,0) ++ (-45: .5cm) to[out=160, in=-60] (0, 0) to[out=120, in =60,loop] () to[out=-120, in=20] ++(225:0.5cm);
 \end{scope}
 
 \coordinate (a) at (90/\order: 4.7); 
 \draw[exceptional] (0:3.5) to (a);
\draw[exceptional] (0:5) to (a);

 \coordinate (b) at (270/\order: 4.7); 
 \draw[exceptional] (360/\order:3.5) to (b);
\draw[exceptional] (360/\order:5) to (b);

 \draw[nonnormal] (300/\order:5cm) to (60/\order:5cm);
\draw[nonnormal] (0:1cm) to (0:6cm);
 \draw[nonnormal] (360/\order:1cm) to (360/\order:6cm);

 \node at (180/\order: 8) {$D$};
\end{scope}
 }

 \foreach \n in {2,3} {
 \begin{scope}[rotate = \n*480/\order]

 \coordinate (a) at (90/\order: 4.7); 
 \coordinate (b) at (270/\order: 4.7); 
 \coordinate (c) at (180/\order: 4.35); 
 
 \draw[exceptional] (0:3.5) to (a);
\draw[exceptional] (0:5) to (a);

\draw[exceptional] (b) to[bend left, looseness = 1.5] (c);

 \draw[exceptional] (360/\order:3.5) to[bend right] (c);
\draw[exceptional] (360/\order:5) to (b);

 \draw[nonnormal] (300/\order:5cm) to (60/\order:5cm);
\draw[nonnormal] (0:1cm) to (0:6cm);
 \draw[nonnormal] (360/\order:1cm) to (360/\order:6cm);

 \node at (180/\order: 8) {$B$};
\end{scope}
}

\begin{scope}
 \foreach \n in {1,2,3} {
 \begin{scope}[rotate = \n*360/\aa+\n*90-30]
\draw[thick, black, fill = black!20!white] 
(0:1cm) 
to (0:6cm) 
  to[out = 90, in =180/\aa+180] 
(180/\aa:10cm +12/\aa cm)
 to[out = 180/\aa+180, in = 360/\aa-90]
 (360/\aa:6cm)
 to (360/\aa:1cm)
  arc [start angle = 360/\aa, end angle = 0, radius = 1cm]
;
\foreach \r in {2cm, 3.5cm, 5cm} 
{
\draw[exceptional]  (0:\r) arc[radius = \r, start angle = 0 , delta angle = 360/\aa] ;
}
\draw[nonnormal] (0:1cm) to (0:6cm);
 \draw[nonnormal] (360/\aa:1cm) to (360/\aa:6cm);

 \end{scope}
 }
\end{scope}

\end{scope}

 \begin{scope}[yshift = -15cm, xshift = 12cm]
\def\aa{12}
 \def\order{4}
 \foreach \n in {1,2,3} {
 \begin{scope}[rotate = \n*360/\order]
\draw[thick, black] 
(0:1cm) 
to (0:6cm) 
  to
(180/\order:10cm)
 to
 (360/\order:6cm)
 to (360/\order:1cm)
  arc [start angle = 360/\order, end angle = 0, radius = 1cm]
;
\foreach \r in {2cm} {
\draw[exceptional]  (0:\r) arc[radius = \r, start angle = 0 , delta angle = 360/\order] ;
}
 \end{scope}
 }
  
 \foreach \n in {1} {
 \begin{scope}[rotate = \n*360/\order]

\begin{scope}[rotate = -90+180/\order, yshift = 4.35cm] 
\draw[exceptional,  every loop/.style={looseness=80, min distance=80}]
 (0,0) ++ (-45: .5cm) to[out=160, in=-60] (0, 0) to[out=120, in =60,loop] () to[out=-120, in=20] ++(225:0.5cm);
 \end{scope}
 
 \coordinate (a) at (90/\order: 4.7); 
 \draw[exceptional] (0:3.5) to (a);
\draw[exceptional] (0:5) to (a);

 \coordinate (b) at (270/\order: 4.7); 
 \draw[exceptional] (360/\order:3.5) to (b);
\draw[exceptional] (360/\order:5) to (b);

 \draw[nonnormal] (300/\order:5cm) to (60/\order:5cm);
\draw[nonnormal] (0:1cm) to (0:6cm);
 \draw[nonnormal] (360/\order:1cm) to (360/\order:6cm);

 \node at (180/\order: 8) {$D$};
\end{scope}
 }

 \foreach \n in {2,3} {
 \begin{scope}[rotate = \n*360/\order]

 \coordinate (a) at (90/\order: 4.7); 
 \coordinate (b) at (270/\order: 4.7); 
 \coordinate (c) at (180/\order: 4.35); 
 
 \draw[exceptional] (0:3.5) to (a);
\draw[exceptional] (0:5) to (a);

\draw[exceptional] (b) to[bend left, looseness = 1.5] (c);

 \draw[exceptional] (360/\order:3.5) to[bend right] (c);
\draw[exceptional] (360/\order:5) to (b);

 \draw[nonnormal] (300/\order:5cm) to (60/\order:5cm);
\draw[nonnormal] (0:1cm) to (0:6cm);
 \draw[nonnormal] (360/\order:1cm) to (360/\order:6cm);

 \node at (180/\order: 8) {$B$};
\end{scope}
}

\begin{scope}
 \foreach \n in {1,2,3} {
 \begin{scope}[rotate = \n*360/\aa-30]
\draw[thick, black, fill = black!20!white] 
(0:1cm) 
to (0:6cm) 
  to[out = 90, in =180/\aa+180] 
(180/\aa:10cm +12/\aa cm)
 to[out = 180/\aa+180, in = 360/\aa-90]
 (360/\aa:6cm)
 to (360/\aa:1cm)
  arc [start angle = 360/\aa, end angle = 0, radius = 1cm]
;
\foreach \r in {2cm, 3.5cm, 5cm} 
{
\draw[exceptional]  (0:\r) arc[radius = \r, start angle = 0 , delta angle = 360/\aa] ;
}
\draw[nonnormal] (0:1cm) to (0:6cm);
 \draw[nonnormal] (360/\aa:1cm) to (360/\aa:6cm);

 \end{scope}
 }
\end{scope}
\end{scope}

\end{tikzpicture}

 \end{figure}
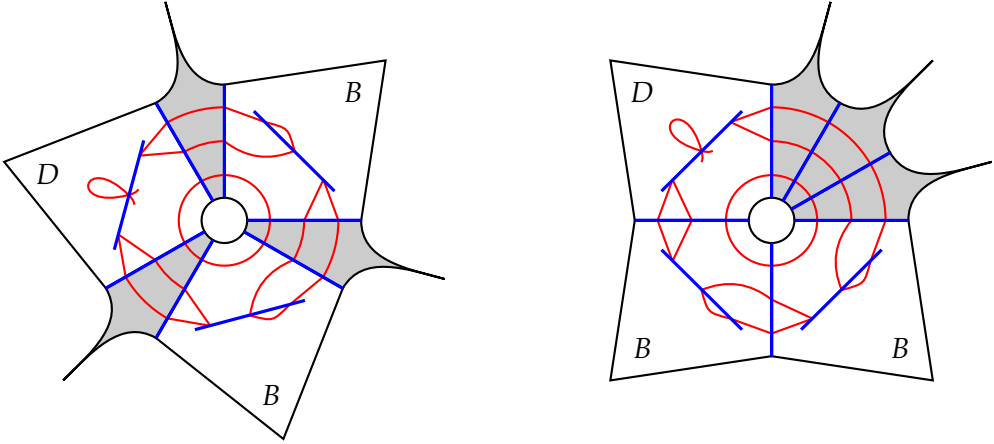

 \begin{cor}\label{cor: infected surface}
 Let $X$ be a stable surface satisfying the conditions of Proposition \ref{prop: infection}. Then for $a>0$ the stable surface $aV \infect X$ does not admit a $\IQ$-Gorenstein smoothing and  has invariants
 \[K_{aV\infect X}^2 = K_X^2 + \frac a3, \quad \chi(aV \infect X)  = \chi(X).\]
 If in addition $X$ is $2$-Gorenstein, then $P_2(aV\infect X) = P_2(X)$.  
 \end{cor}
\begin{proof}
 Using the formulas from Proposition \ref{prop: easy glueing} we see that
 \[K_{aV\infect X}^2 =  K_X^2 + a (K_V+C_1+C_2)^2  = K_X^2 + \frac a3.\]
For the holomorphic Euler characteristic we compute
\[ \chi(aV\infect X) -\chi(X) = a \left(\chi( V) -\frac 12 \chi(C_1+C_2) -\frac 12 \cdot 3\right) = a\left( 1+\frac 12 - \frac 32\right)= 0.\]
The surface $aV\infect X$ has $3a$ quotient singularities of type $\frac13(1,1)$ on the components isomorphic to $V$. Since such a singularity does not admit a $\IQ$-Gorenstein smoothing by \cite[Sect.~3]{ksb88}, the surface cannot admit a $\IQ$-Gorenstein smoothing. 

To compute the second plurigenus we consider a different partial resolution $\pi\colon a\hat V \infect X\to a V\infect X$ where we resolve all quotient singularites, so $a\hat V \infect X$ contains $a$ exceptional $-3$ curves $E_i$.
It is easy to compute locally on $V$ to confirm that 
\[\pi^*2K_{aV\infect X} = 2K_{a\hat V\infect X} + \sum_{i=1}^a \frac23 E_i.\]
Thus $\pi_* \omega_{a\hat V\infect X}^{[2]} = \omega_{aV\infect X}^{[2]}$ and since quotient singularites are rational we have
\[ H^i(\omega_{aV\infect X}^{[2]}) = H^i(\omega_{a\hat V\infect X}^{[2]}).\]
For $i>0$ these groups vanish by the generalised Kodaira vanishing, 
thus we can apply the Riemann-Roch formula \cite{LR16}  to the $2K_{a\hat V \infect X}$, which is Cartier by our assumptions, to compute
\begin{align*}
P_2(a V\infect X) &= \chi(\ko_{a \hat V \infect X}) + K_{a\hat V\infect X}^2 = \chi(X) + \left(\pi^*K_{aV\infect X} -\sum_{i=1}^a \frac 13 E_i\right)^2\\
&= \chi(X) + K_X^2 + \frac a3  +0 -a \frac 39  =  \chi(X) + K_X^2 = P_2(X)
 \end{align*}
By Corollary \ref{cor: non-smoothable} this gives another proof that $X$ is not smoothable, even if the invariants are integral.
\end{proof}

\subsection{Infected examples and  pure virus surfaces: proof of Theorem A}
We prove Theorem A by exibiting for each classically possible invariant a stable surface that is not $\IQ$-Gorenstein smoothable and not even in the same  connected component as a smoothable surface. Most of them come from the infection mechanism, but the first are pure virus surfaces.

 \begin{exam}[Pure virus surfaces]\label{exam: pure}
 We construct three surfaces from three copies of $V$ each as depicted in Figure \ref{fig: pure virus volume 1}. Then from Proposition \ref{prop: easy glueing} and arguing as in Corollary \ref{cor: infected surface} we see that
 \[K_{Z_{1, l}}^2 = 1, \quad P_2(Z_{1,l}) = \chi(Z_{1, l}) = l,\] and that $Z_{1,l}$ does not admit a $\IQ$-Gorenstein smoothing in the large.

\begin{figure}[h!]\caption{Pure triple virus surfaces (partial resolution)}\label{fig: pure virus volume 1}
\scriptsize
 \begin{tikzpicture}
[curves/.style = {thick},
exceptional/.style = {red, thick}, 
nonnormal/.style ={very thick, blue},
boundary/.style={nonnormal, dashed},
scale = .25
]

\begin{scope}[xshift = 12cm, yshift = -18cm] %
 \def\order{3}
 \node at (120/\order : 10cm +12/\order cm) {$Z_{1,1}$};
 \foreach \n in {1,2,3} {
 \begin{scope}[rotate = \n*360/\order]
\draw[thick, black, fill = black!20!white] 
(0:1cm) 
to (0:6cm) 
  to[out = 90, in =180/\order+180] 
(180/\order:10cm +12/\order cm)
 to[out = 180/\order+180, in = 360/\order-90]
 (360/\order:6cm)
 to (360/\order:1cm)
  arc [start angle = 360/\order, end angle = 0, radius = 1cm]
;
 \draw[nonnormal] (0:1cm) to (0:6cm);
 \draw[nonnormal] (360/\order:1cm) to (360/\order:6cm);
 \end{scope}
 }
\foreach \r in {2cm, 3.5cm, 5cm} 
{
\draw[exceptional]  (0:\r) arc[radius = \r, start angle = 0 , delta angle = 2*360/\order] ;
}
\draw[exceptional] (0:2)  to[out = -90, in = 330]  (-360/\order:3.5) ;
\draw[exceptional] (0:3.5) to[out = -90, in = 330]  (-360/\order:5) ;

\draw[ black!20!white,  line width = .15cm] (0:5)++(-90:0.1) to[out = -90, in = 330]  (-360/\order+5:2) ;
\draw[exceptional] (0:5) to[out = -90, in = 330]  (-360/\order:2) ;
 \end{scope}

\begin{scope}[xshift = 24cm] 
 \def\order{3}
  \node at (120/\order : 10cm +12/\order cm) {$Z_{1,2}$};
 \foreach \n in {1,2,3} {
 \begin{scope}[rotate = \n*360/\order]
\draw[thick, black, fill = black!20!white] 
(0:1cm) 
to (0:6cm) 
  to[out = 90, in =180/\order+180] 
(180/\order:10cm +12/\order cm)
 to[out = 180/\order+180, in = 360/\order-90]
 (360/\order:6cm)
 to (360/\order:1cm)
  arc [start angle = 360/\order, end angle = 0, radius = 1cm]
;
 \draw[nonnormal] (0:1cm) to (0:6cm);
 \draw[nonnormal] (360/\order:1cm) to (360/\order:6cm);
 \end{scope}
 }
\foreach \r in {2cm, 3.5cm, 5cm} 
{
\draw[exceptional]  (0:\r) arc[radius = \r, start angle = 0 , delta angle = 2*360/\order] ;
}
\draw[exceptional] (0:2) arc[radius = 2, start angle = 0 , delta angle = -360/\order] ;
\draw[exceptional] (0:3.5) to[out = -90, in = 330]  (-360/\order:5) ;

 \draw[ black!20!white,  line width = .15cm] (-5:4.5) to[out = -90, in = 330]  (-360/\order+5:3.5) ;
\draw[exceptional] (0:5) to[out = -90, in = 330]  (-360/\order:3.5) ;
 \end{scope}

\begin{scope}
 \def\order{3} \node at (120/\order : 10cm +12/\order cm) {$Z_{1,3}$};
 \foreach \n in {1,2,3} {
 \begin{scope}[rotate = \n*360/\order]
\draw[thick, black, fill = black!20!white] 
(0:1cm) 
to (0:6cm) 
  to[out = 90, in =180/\order+180] 
(180/\order:10cm +12/\order cm)
 to[out = 180/\order+180, in = 360/\order-90]
 (360/\order:6cm)
 to (360/\order:1cm)
  arc [start angle = 360/\order, end angle = 0, radius = 1cm]
;
\foreach \r in {2cm, 3.5cm, 5cm} 
{
\draw[exceptional]  (0:\r) arc[radius = \r, start angle = 0 , delta angle = 360/\order] ;
}
\draw[nonnormal] (0:1cm) to (0:6cm);
 \draw[nonnormal] (360/\order:1cm) to (360/\order:6cm);

 \end{scope}
 }
\end{scope}

\end{tikzpicture}
\end{figure}
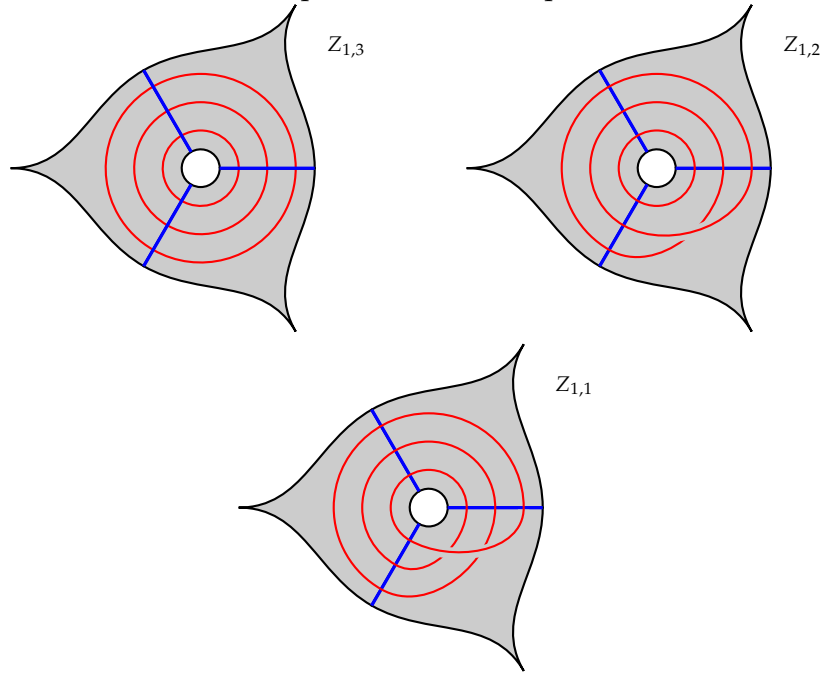
\end{exam}

\begin{rem}\label{rem: coughlan}
 Together with Stephen Coughlan we observed that $Z_{1,3}$ can be realised as a hypersurface of degree $9$ in $\IP(1,1,3,3)$ given by the product of three general weighted homogeneous polynomials of degree $3$. Then $\omega_{Z_{1,3}}=\ko_{Z_{1,3}}(1)$, so $P_2(Z_{1,3}) = 2$. This was the motivating example to look at the plurigenera of the general $aV\infect X$. 
\end{rem}

 \begin{exam}[Infecting surfaces $X_{k,l}$]\label{exam: infect}
Let $l> 0$. Then as long as $k-a+2\geq l$, that is, $1\leq a\leq k-l+2$ for some positive integer $a$ we constructed a surface $X_{k-a, l}$ in Section \ref{sect: Xkl}. By Corollary \ref{cor: infected surface} the infected surface $Z_{k,l} = 3aV\infect X_{k-a,l}$ is a stable surface with invariants $K_{Z_{k,l}}^2 = k$ and $\chi(Z_{k,l}) = l$ which does not admit a $\IQ$-Gorenstein smoothing in the large.
 \end{exam}
 
Considering the inequality $\chi(S)>0$ and the Noether inequality $K_S^2\geq2 \chi(S)-6$ for minimal surfaces of general type, we see that Example \ref{exam: infect} covers all classical invariants with the exception of  the cases with $K_S^2 = 1$ and $\chi(S) = 1,2,3$, which have been constructed in Example \ref{exam: pure}

Since they have different second plurigenus, all the examples we constructed cannot share a connected component with a component of the Gieseker moduli space. 
This concludes the proof of Theorem $A$. \hfill\qed

\section{Final observations}
The simplicity of our construction suggests that we should in general expect many new irreducible and connected components in the moduli space of stable surfaces. Indeed, the surfaces
$ 3aV \infect X_{k-a, l} $ for $0\leq a\leq k-l+2$ all have the same main invariants but different second plurigenus by Proposition \ref{prop: infection}, so by Corollary \ref{cor: non-smoothable} the moduli space $\overline \gothM_{k,l}$ has at least $k-l+2$ connected components. 

In addition, instead of $V$ one could try to repeat the construction with  any non-$\IQ$-Gorenstein smoothable stable surface which admits a suitable boundary, for example, a weighted projective plane containing a curve of low genus not passing through the singular points. 

Also, in the construction of the surfaces $X_{k,l}$ there is still a lot of flexibility and choice giving the same invariants, except in the extremal cases. Possibly, this would enable us to realise also different values for geometric genus and irregularity, which could be computed using the methods from \cite{FPR15b}.

\end{document}